\documentclass[12pt,twoside]{amsart}
\usepackage{amssymb,amsrefs,latexsym,amsmath,amsthm,graphicx}

\newtheorem{lemma}{Lemma}

\newtheorem{theorem}[lemma]{Theorem}

\theoremstyle{definition}

\newtheorem{remark}[lemma]{Remark}
\newtheorem{definition}[lemma]{Definition}

\newcommand{\Diff}{\mathrm{Diff}}
\newcommand{\Ham}{\mathrm{Ham}}
\newcommand{\id}{\mathrm{id}}
\newcommand{\Iso}{\mathrm{Iso}}
\newcommand{\osc}{\mathrm{osc\, }}
\newcommand{\R}{\mathbb{R}}
\newcommand{\supp}{\mathrm{supp \, }}
\newcommand{\Symp}{\mathrm{Symp}}

\begin{document}
\thispagestyle{plain}

\title{The symplectic displacement energy}

\author{Augustin Banyaga$^\dagger$}
\let\thefootnote\relax\footnote{$^\dagger$Corresponding author: banyaga@math.psu.edu}
\address{Department of Mathematics \\
         Penn State University \\
         University Park, PA 16802}
\email{banyaga@math.psu.edu}

\author{David E. Hurtubise}
\address{Department of Mathematics and Statistics\\
         Penn State Altoona\\
         Altoona, PA 16601-3760}
\email{Hurtubise@psu.edu}

\author{Peter Spaeth}
\address{Department of Mathematics and Statistics\\
         Penn State Altoona\\
         Altoona, PA 16601-3760}
\email{spaeth@psu.edu}

\keywords{Symplectic displacement energy, Hofer-like metric, Hofer metric, Calabi invariant, flux, strong symplectic homeomorphism}

\subjclass[2010]{Primary: 53D35 Secondary: 57R17}

\begin{abstract}
We define the symplectic displacement energy of a non-empty subset of a compact symplectic
manifold as the infimum of the Hofer-like norm \cite{BanAHo} of symplectic diffeomorphisms that
displace the set. We show that this energy (like the usual displacement energy defined using
Hamiltonian diffeomorphisms) is a strictly positive number on sets with non-empty interior.
As a consequence we prove a result justifying the introduction of the notion of strong symplectic homeomorphisms \cite{BanOnt}.
\end{abstract}

\maketitle

%%%%%%%%%%%%%%%%%%%%%%%%%%%%%%%%%%%%%%%%%%%%%%%%%%%%%%%%%

\section{Statement of results}

In \cite{HofOnt}, Hofer defined a norm $\| \cdot \|_H$ on the group $\Ham(M,\omega)$ of compactly supported Hamiltonian diffeomorphisms of a symplectic manifold $(M,\omega)$.

For a non-empty subset $A\subset M$, he introduced the notion of the {\bf displacement energy} $e(A)$ of A:
	\[ e(A) = \inf\{ \|\phi\|_H \mid \phi \in \Ham(M,\omega), \phi(A)\cap A =\emptyset\}. \]
The displacement energy is defined to be $+\infty$ if no compactly supported Hamiltonian diffeomorphism displaces $A$.

Eliashberg and Polterovich \cite{EliBii} proved the following result.

\begin{theorem}\label{theorem:eliashberg-polterovich}
For any non-empty open subset $A$ of $M$, $e(A)$ is a strictly positive number.
\end{theorem}

It is easy to see that if $A$ and $B$ are non-empty subsets of $M$ such that 
$A \subset B$, then $e(A) \leq e(B)$, and that $e$ is a symplectic invariant.
That is,
	\[ e(f(A)) = e(A) \]
for all $f \in \Symp(M,\omega) = \{\phi \in \Diff(M) \mid \phi^\ast \omega = \omega \}$.
This follows  from the fact that $\|f \circ \phi \circ f^{-1}\|_H = \|\phi\|_H$.
 
In \cite{BanAHo}, a Hofer-like metric $\| \cdot \|_{HL}$ was constructed on the group $\Symp_0(M,\omega)$
of all symplectic diffeomorphisms of a compact symplectic manifold $(M,\omega)$ that are isotopic to
the identity. It was proved recently by Buss and Leclercq \cite{BusPse} that the restriction of
$\| \cdot \|_{HL}$ to $\Ham(M,\omega)$ is a metric equivalent to the Hofer metric.

\smallskip
Let us now propose the following definition.

\begin{definition}\label{symplecticdisplacement}
The {\bf symplectic displacement energy} $e_s(A)$ of a non-empty subset $A \subset M$ is defined
to be:
	\[ e_s(A) = \inf\{ \|h\|_{HL} \mid h \in \Symp_0(M,\omega),  h(A)\cap A = \emptyset\} \]
if some element of $\Symp_0(M,\omega)$ displaces $A$, and $+\infty$ if no element of
$\Symp_0(M,\omega)$ displaces $A$.
\end{definition}

\noindent
Clearly, if $A$ and $B$ are non-empty subsets of $M$ such that $A\subset B$, then 
$e_s(A) \leq e_s(B)$.

\smallskip
The goal of this paper is to prove the following result.

\begin{theorem}\label{theorem:displacement-positive}
For any closed symplectic manifold $(M,\omega)$, the symplectic displacement energy of any subset $A \subset M$ with non-empty interior satisfies $e_s(A)>0$.
\end{theorem}

%%%%%%%%%%%%%%%%%%%%%%%%%%%%%%%%%%%%%%%%%%%%%%%%%%%%%%%

\section{The Hofer norm $\| \cdot \|_H$ and the Hofer-like norm $\| \cdot \|_{HL}$} 

\subsection{Symp$_\mathbf{0}\mathbf{(M,\omega)}$ and Ham$\mathbf{(M,\omega)}$}
Let $\Iso(M,\omega)$ be the set of all compactly supported symplectic isotopies of a symplectic manifold $(M,\omega)$.
A compactly supported symplectic isotopy $\Phi \in \Iso(M,\omega)$ is a smooth map $\Phi: M \times [0,1] \to M$
such that for all $t \in [0,1]$, if we denote by $\phi_t(x) = \Phi(x,t)$, then $\phi_t \in 
\Symp(M,\omega)$ is a symplectic diffeomorphism
with compact support, and $\phi_0 = \id$. We denote by $\Symp_0(M,\omega)$ the set of all time-$1$ maps of compactly supported
symplectic isotopies.

Isotopies  $\Phi = \{ \phi_t \}$ are in one-to-one correspondence
with families of smooth vector fields $\{\dot{\phi}_t\}$ defined by
	\[  \dot {\phi}_t (x) = \frac {d \phi_t}{dt} (\phi_t^{-1}(x)). \]
If $\Phi \in \Iso(M,\omega)$, then the one-form $i(\dot {\phi}_t )\omega$ such that
	\[ i(\dot {\phi}_t )\omega (X) = \omega (\dot {\phi}_t,X) \]
for all vector fields $X$ is closed. 
If for all $t$ the 1-form $i(\dot {\phi}_t )\omega$ is exact, that is, there exists a smooth
function $F:M \times [0,1] \rightarrow \mathbb{R}$, $F(x,t)= F_t(x)$, with compact supports such
that $i(\dot {\phi}_t )\omega = dF_t$, then the isotopy $\Phi$ is called a Hamiltonian isotopy and will be denoted by $\Phi_F$. We define the group $\Ham(M,\omega)$
of Hamiltonian diffeomorphisms as the set of time-one maps of Hamiltonian isotopies.

For each $\Phi = \{ \phi_t\} \in \Iso(M,\omega)$, the mapping
	\[ \Phi \mapsto \left[ \int_0^1 (i(\dot {\phi}_t) \omega)dt \right], \]
where $[\alpha]$ denotes the cohomology class of a closed form $\alpha$, induces a well defined map
$\tilde {S}$ from the universal cover of $\Symp_0(M,\omega)$ to the first de Rham cohomology group
$H^1(M,\mathbb{R})$. This map is called the {\bf Calabi invariant} (or the {\bf flux}). It is a surjective group homomorphism.
Let $\Gamma \subset H^1(M,\mathbb{R})$ be the image by $\tilde S$ of the fundamental group of
$\Symp_0(M,\omega)$. We then get a surjective homomorphism
	\[ S: \Symp_0(M,\omega) \to H^1(M,\mathbb{R})/{\Gamma}. \]
The kernel of this homomorphism is the group $\Ham(M,\omega)$ \cites{BanSur, BanThe}.

%--------------------------------------

\subsection{The Hofer norm}
Hofer \cite{HofOnt} defined the length $l_H$ of a Hamiltonian isotopy $\Phi_F$ as
	\[ l_H(\Phi_F)  =  \int_0^1 (\osc F(x,t))\ dt, \]
where the oscillation of a function $f:M \to \R$ is 
	\[ \osc(f) = \max_{x \in M}(f)- \min_{x \in M}(f). \]
For $\phi \in \Ham(M,\omega)$, the {\bf Hofer norm} of $\phi$ is
	\[ \|\phi\|_H = \inf \{ l_H(\Phi_F)\}, \]
where the infimum is taken over all Hamiltonian isotopies $\Phi_F$ with time-one map equal to $\phi$, i.e.\ $\phi_{F,1} = \phi$.

The Hofer distance $d_H(\phi, \psi)$ between two Hamiltonian diffeomorphisms $\phi$ and $\psi$ is
\begin{align*}
d_H(\phi, \psi) = \|\phi \circ \psi^{-1}\|_H.
\end{align*}
This distance is bi-invariant.
This property was used in \cite{EliBii} to prove Theorem \ref{theorem:eliashberg-polterovich}.

%---------------------------------------

\subsection{The Hofer-like norm}
Now let $(M,\omega)$ be a compact symplectic manifold without boundary, on which we fix a Riemannian metric $g$.
For each $\Phi = \{ \phi_t \} \in \Iso(M,\omega)$, we consider the Hodge decomposition
\cite{WarFou} of the $1$-form $i(\dot {\phi}_t) \omega$ as
	\[ i(\dot {\phi}_t) \omega  =  \mathcal{H}_t  + du_t, \]
where $\mathcal{H}_t$ is a harmonic $1$-form.
The forms $\mathcal{H}_t$ and $u_t$ are unique and depend smoothly on $t$.

For $\Phi \in \Iso(M,\omega)$, define
	\[ l_0 (\Phi)  =  \int_0^1 (|\mathcal{H}_t| + \osc(u(x,t)))\ dt, \]
where $|\mathcal{H}_t|$ is a norm on the finite dimensional vector space of 
harmonic 1-forms.
We let
	\[ l(\phi)  = \frac{1}{2}( l_0(\Phi) + l_0(\Phi^{-1})), \]
where $\Phi^{-1} = \{ \phi_t^{-1} \}$.

For each $\phi \in \Symp_0(M,\omega)$, let
	\[ \|\phi\|_{HL}  =   \inf \{l(\phi)\}, \]
where the infimum is taken over all symplectic isotopies $\Phi$ with $\phi_1 = \phi$.

\smallskip\noindent
The following result was proved in \cite{BanAHo}.

\begin{theorem}
For any closed symplectic manifold $(M,\omega)$, $\| \cdot \|_{HL}$ is a norm on
$\Symp_0(M,\omega)$.
\end{theorem}

\begin{remark}
The norm $\| \cdot \|_{HL}$ depends on the choice of the Riemannian metric $g$ on $M$ and the
choice of the norm $|\cdot |$ on the space of harmonic 1-forms. However, different
choices for $g$ and $|\cdot |$ yield equivalent metrics.  See Section 3 of \cite{BanAHo} for more details.
\end{remark}

%-------------------------------

\subsection{Some equivalence properties}

Let $(M,\omega)$ be a compact symplectic manifold.  Buss and Leclercq have proved:
\begin{theorem}\cite{BusPse}
The restriction of the Hofer-like norm $\|\cdot\|_{HL}$ to $\Ham(M,\omega)$ is 
equivalent to the Hofer norm $\|\cdot\|_H$.
\end{theorem}

We now prove the following.

\begin{theorem}\label{conjinequality}
Let $\phi \in \Symp_0(M,\omega)$. The norm
$$
h \mapsto \| \phi \circ h \circ \phi^{-1}\|_{HL}
$$
on $\Symp_0(M,\omega)$ is equivalent to the norm $\|\cdot\|_{HL}$.
\end{theorem}

\begin{remark}
We owe the statement of the above theorem to the referee of a previous version
of this paper.
\end{remark}

\begin{proof}
Let $\{h_t\}$ be an isotopy in $\Symp_0(M,\omega)$ from h to the identity, and let
\[ i(\dot {h}_t) \omega  =  \mathcal{H}_t  + du_t, \]
be the Hodge decomposition.  Then $\Psi = \{\phi \circ h_t \circ \phi^{-1}\}$ is an
isotopy from $\phi \circ h \circ \phi^{-1}$ to the identity and $\dot{\Psi}_t = 
\phi_\ast \dot{h}_t$. Therefore,
\[ 
i(\dot{\Psi}_t)\omega = (\phi^{-1})^\ast(i(\dot{h}_t)\phi^\ast\omega) = 
(\phi^{-1})^\ast(\mathcal{H}_t + du_t)
= (\phi^{-1})^\ast \mathcal{H}_t + d(u_t \circ \phi^{-1}).
\]

Let $\{\phi_s^{-1}\}$ be an isotopy from $\phi^{-1}$ to the identity, and let
$L_X = i_X d + d i_X$ be the Lie derivative in the direction $X$. Then
\[
\frac{d}{ds}((\phi^{-1}_s)^\ast \mathcal{H}_t) = 
(\phi^{-1}_s)^\ast(L_{\dot{\phi}^{-1}_s} \mathcal{H}_t) = 
d ((\phi^{-1}_s)^\ast i(\dot{\phi}_s^{-1})\mathcal{H}_t),
\]
where $\dot{\phi}_t^{-1}= (\frac{d}{dt} \phi_t^{-1}) \circ \phi_t$.
Integrating from $0$ to $1$ we get
\[
(\phi^{-1})^\ast \mathcal{H}_t - \mathcal{H}_t = d \alpha_t
\]
where
\[
\alpha_t = \int_0^1 ((\phi^{-1}_s)^\ast i(\dot{\phi}_s^{-1})\mathcal{H}_t)\ ds.
\]
Therefore,
\[
i(\dot{\Psi}_t)\omega = \mathcal{H}_t + d(u_t \circ \phi^{-1} + \alpha_t).
\]
Hence,
\begin{eqnarray*}
l_0(\Psi) & = & \int_0^1 \left( |\mathcal{H}_t| + \osc(u_t \circ \phi^{-1} + \alpha_t)\right) dt\\
& \leq & \int_0^1 \left( |\mathcal{H}_t| + \osc(u_t \circ \phi^{-1}) \right) dt + \int_0^1 
         \osc(\alpha_t)\ dt\\
& = & \int_0^1 \left( |\mathcal{H}_t| + \osc(u_t) \right) dt + \int_0^1 \osc (\alpha_t)\ dt\\
& = & l_0(\{h_t\}) + K
\end{eqnarray*}
where
\[
K = \int_0^1 \osc (\alpha_t)\ dt.
\]

\smallskip
Let us now do the same calculation for $\Psi^{-1} = \{\phi \circ h_t^{-1} \circ \phi^{-1}\}$.

\medskip
Since $\dot{h}_t^{-1}$ satisfies $\dot{h}^{-1}_t = -(h_t^{-1})_\ast \dot{h}_t$, the 
cohomology classes of $i(\dot{h}_t)\omega$ and $i(\dot{h}_t^{-1})\omega$ are opposite.
Since the Hodge decomposition is unique and the harmonic part of the first form is 
$\mathcal{H}_t$, the harmonic part of the second form is $-\mathcal{H}_t$. Therefore, there
is a smooth family of functions $v_t$ such that the Hodge decomposition for
$i(\dot{h}_t^{-1})\omega$ is
\[
i(\dot{h}_t^{-1})\omega = - \mathcal{H}_t + d v_t.
\]
The same calculation shows 
\[
i(\dot{\Psi}_t^{-1})\omega = -\mathcal{H}_t + d(v_t \circ \phi^{-1} - \alpha_t).
\]
Hence,
\[
l_0(\Psi^{-1}) \leq l_0(\{h_t^{-1}\}) + K.
\]

\bigskip
We will now estimate $K = \int_0^1 \osc (\alpha_t)\ dt$. Fix an isotopy $\{\phi_s^{-1}\}$ from 
$\phi^{-1}$ to the identity. Consider the continuous linear map
\[
\mathcal{L}_{\{\phi_s^{-1}\}}: \mathcal{H}^1(M,g) \rightarrow C^\infty(M)
\]
from the finite dimensional vector space of harmonic $1$-forms
given by
\[
\mathcal{L}_{\{\phi_s^{-1}\}}(\theta) = \int_0^1 ((\phi^{-1}_s)^\ast i(\dot{\phi}_s^{-1})\theta)\ ds.
\]
Let $\nu \geq 0$ be the norm of $\mathcal{L}_{\{\phi_s^{-1}\}}$ where the norm on 
$\mathcal{H}^1(M,g)$ is defined by the metric $g$ and $C^\infty(M)$ is given the sup norm.  Then
$|\mathcal{L}_{\{\phi_s^{-1}\}}(\theta)| \leq \nu |\theta|$.  In our case $\alpha_t = 
\mathcal{L}_{\{\phi_s^{-1}\}}(\mathcal{H}_t)$.
Therefore,
\[
|\alpha_t| \leq \nu |\mathcal{H}_t|
\]
and
\[
\osc(\alpha_t) \leq 2 |\alpha_t| \leq 2 \nu |\mathcal{H}_t|.
\]
This implies
\[
\osc(\alpha_t) \leq 2 \nu \left( |\mathcal{H}_t| + \osc(u_t) \right)
\text{ and }\
\osc(\alpha_t) \leq 2 \nu \left( |\mathcal{H}_t| + \osc(v_t) \right).
\]
Hence,
\[
K = \int_0^1 \osc(\alpha_t)\ dt \leq 2 \nu\, l_0(\{h_t\}),
\]
and
\[
K = \int_0^1 \osc(\alpha_t)\ dt \leq 2 \nu\, l_0(\{h_t^{-1}\}).
\]

\bigskip
Now recall that, 
\[
l_0(\Psi) \leq l_0(\{h_t\}) + K  \ \text{ and }\ l_0(\Psi^{-1}) \leq l_0(\{h_t^{-1}\}) + K.
\]
Therefore,
\begin{eqnarray*}
l(\Psi) & = & \frac{1}{2}\left(l_0(\Psi)+ l_0(\Psi^{-1})\right)\\
& \leq & \frac{1}{2}\left(l_0(\{h_t\}) + 2\nu\, l_0(\{h_t\}) + l_0(\{h_t^{-1}\}) +
         2\nu\, l_0(\{h_t^{-1}\}) \right)\\
&\leq & (2\nu + 1) l(\{h_t\}).
\end{eqnarray*}
Taking the infimum over the set $I(h)$ of all symplectic isotopies
from $h$ to the identity we get
\[
\inf_{I(h)} l(\Psi) \leq (2\nu + 1) \|h\|_{HL},
\]
and since 
\[
\|\phi \circ h \circ \phi^{-1}\|_{HL} \leq \inf_{I(h)} l(\Psi)
\]
we get 
\[
\|\phi \circ h \circ \phi^{-1} \|_{HL} \leq k \|h\|_{HL}
\]
with $k = 2\nu + 1$.

\medskip
We have shown that for every $\phi \in \Symp_0(M,\omega)$ there is a $k \geq 1$ (depending
on an isotopy $\{\phi_s\}$ from $\phi$ to the identity) such that the preceding 
inequality holds for all $h\in \Symp_0(M,\omega)$.
Applying this to $\phi^{-1}$ we see that there is an $k' \geq 1$ such that
\[
\|\phi^{-1} \circ h \circ \phi \|_{HL} \leq k' \|h\|_{HL}
\]
for all $h \in \Symp_0(M,\omega)$. Therefore, for any $h \in \Symp_0(M,\omega)$
we have 
\[
\|h\|_{HL} = \|\phi^{-1} \circ (\phi \circ h \circ \phi^{-1}) \circ \phi\|_{HL}
\leq k' \|\phi \circ h \circ \phi^{-1}\|_{HL}.
\]
That is,
\[
\frac{1}{k'} \|h\|_{HL} \leq \|\phi \circ h \circ \phi^{-1}\|_{HL} \leq k \|h\|_{HL}.
\]

\end{proof}

\begin{remark}
If $\phi$ is Hamiltonian and contained in a Weinstein chart, then
there is canonical isotopy $\{\phi_s\}$ with $\phi_1 = \phi$.  In this
case $k$ can be viewed as depending only on $\phi$.
\end{remark}

%%%%%%%%%%%%%%%%%%%%%%%%%%%%%%%%%%%%%%%%%%%%%%%%%%%%%%%%%%%%

\section{Proof of the main result}
We will closely follow the proof given by Polterovich of Theorem 2.4.A in \cite{PolThe} that $e(A)>0$.
We will use without any change Proposition 1.5.B. 

\smallskip\noindent
{\bf Proposition 1.5.B.} \cite{PolThe} {\em
For any non-empty open subset $A$ of $M$, there exists a pair of Hamiltonian diffeomorphisms
$\phi$ and $\psi$ that are supported in $A$ and whose commutator $[\phi, \psi] = \psi^{-1}
\circ \phi^{-1} \circ \psi \circ \phi$ is not equal to the identity.}

\smallskip
For the sake of completeness we provide the following alternate proof of this proposition based on the transitivity lemmas in \cite{BanThe} (pages 29 and 109).
(For a proof of $k$-fold transitivity for symplectomorphisms see \cite{BooTra}.)

\begin{proof}
Let $U$ be an open subset of A such that $\overline U \subset A$. Pick three distinct points $a,b,c \in U$.
By the transitivity lemma of $\Ham(M,\omega)$, there exist $\phi, \psi \in \Ham(M,\omega)$ such that 
$\phi(a) = b$ and $\psi(b) = c$.  Moreover, we can choose $\phi$ and $\psi$ so that $\supp(\phi)$ and $\supp(\psi)$ are contained in small tubular neighborhoods $V$ and $W$ of distinct paths in $U$ joining $a$ to $b$ and
$b$ to $c$ respectively, and we can assume that $c \in U\backslash V$.
\begin{figure}[h]
\includegraphics{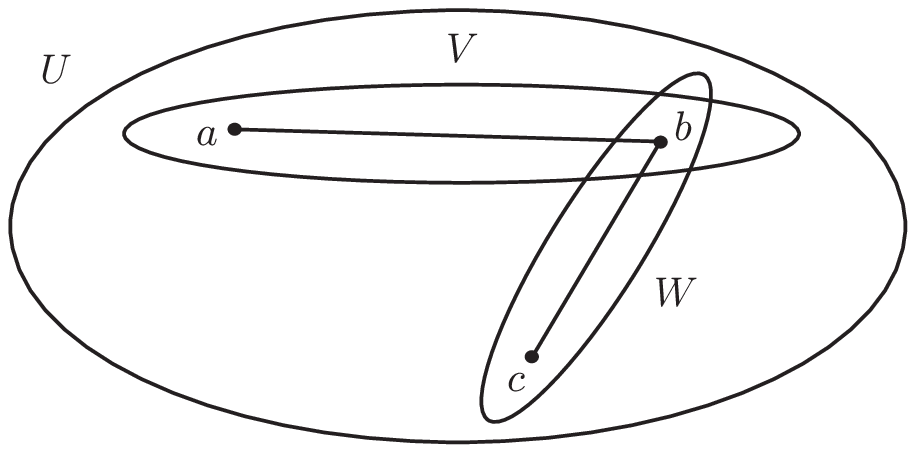}
\end{figure}

\noindent
Then $(\psi^{-1}\phi^{-1}\psi\phi)(a) = (\psi^{-1}\phi^{-1})(c) = \psi^{-1}(c) = b$.
Hence $[\phi, \psi] \neq $ id.

\end{proof}

\smallskip

We will say that a map $h$ \textbf{displaces} $A$ if $h(A) \cap A = \emptyset$.  Let us denote by $D(A)$ the set of all $h \in \Symp_0(M,\omega)$ that displace $A$. We note the following fact.  

\begin{lemma}
Let $\phi$ and $\psi$ be as in Proposition 1.5.B, and let 
$h \in D(A)$.  Then the commutator 
 \[
  \theta = [h, \phi^{-1}] = \phi\circ h^{-1}\circ \phi^{-1}\circ h
 \]
satisfies $[\phi,\psi] = [\theta,\psi]$.
\end{lemma}

\begin{proof}
If $x \in A$ then $h(x) \not\in A$. Hence,
\begin{eqnarray*}
\theta(x) & = & (\phi \circ h^{-1})(\phi^{-1}(h(x)))\\
          & = & \phi(h^{-1}(h(x))) \qquad \text{since } \supp(\phi^{-1}) \subset A\\
          & = & \phi(x),
\end{eqnarray*}
and we see that $\theta|_A = \phi|_A$.  Similarly, for $x \in A$ we have $\phi^{-1}(x) \in A$,
and hence $h(\phi^{-1}(x)) \not\in A$ since $h(A) \cap A = \emptyset$.
Thus,
\begin{eqnarray*}
\theta^{-1}(x) & = & h^{-1}(\phi(h(\phi^{-1}(x))))\\
               & = & h^{-1}(h(\phi^{-1}(x))) \qquad \text{since } \supp(\phi) \subset A\\
               & = & \phi^{-1}(x),
\end{eqnarray*}
and we see that $\theta^{-1}|_A = \phi^{-1}|_A$.
Thus, $(\phi^{-1}\circ \psi\circ \phi)(x) = (\theta^{-1}\circ \psi\circ \theta)(x)$
for all $x \in A$ since $\supp(\psi) \subset A$. 

Now, if $x \not\in A$ and $\theta(x) \in A$ we would have 
$x = \theta^{-1}(\theta(x)) = \phi^{-1}(\theta(x)) \in A$ since $\supp(\phi^{-1})\subset A$,
a contradiction. Hence, for $x \not\in A$ we have $\theta(x) \not\in A$ and  
 \[
(\phi^{-1}\circ \psi\circ \phi)(x) = x = (\theta^{-1}\circ \psi\circ \theta)(x)
 \]
since both $\phi$ and $\psi$ have support in $A$.
Therefore, $\phi^{-1}\circ \psi\circ \phi = \theta^{-1}\circ \psi\circ \theta$, and we have $[\phi, \psi] = [\theta, \psi]$.

\end{proof}

\begin{proof}[Proof of Theorem~\ref{theorem:displacement-positive} continued]

Following the proof of Theorem 2.4.A in \cite{PolThe} we assume there exists $ h \in D(A) \neq \emptyset$. Otherwise, we are done since 
$e_s(A) = +\infty$. 
The commutator $\theta$ is contained in $\Ham(M,\omega)$ because commutators are in the kernel of the Calabi invariant.
Since both $\theta$ and $\psi$ are in $\Ham(M,\omega)$ and the Hofer norm is conjugation invariant, we have
\begin{eqnarray*}
\|[\theta, \psi]\|_H & = & \|\psi^{-1}\circ \theta^{-1} \circ \psi \circ \theta\|_H\\
& \leq & \|\psi^{-1} \circ \theta^{-1} \circ \psi\|_H + \|\theta\|_H\\
& = & 2\|\theta\|_H. 
\end{eqnarray*}
By Buss and Leclercq's theorem \cite{BusPse} there is constant $\lambda > 0$ such that
\begin{align*}
\|\theta\|_H \leq \lambda \| \theta\|_{HL}. 
\end{align*}
Using the triangle inequality and the symmetry of $\| \cdot \|_{HL}$ we have
\begin{eqnarray*}\label{inequality:hofer-commutator}
\|[\theta,\psi]\|_{H} & \leq & 2\lambda \left(  \| \phi \circ h \circ \phi^{-1}\|_{HL} + \|h\|_{HL}\right)\\
& \leq & 2\lambda \left(k \|h\|_{HL} + \|h\|_{HL}\right),
\end{eqnarray*}
where $k > 0$ is the constant given by Theorem \ref{conjinequality}.
Therefore,
$$
0 <\frac{ \|[\theta,\psi]\|_{H}}{2\lambda(k+1)} \leq \|h\|_{HL}.
$$
Since this inequality holds for all $h \in D(A)$ we can take the
infimum over $D(A)$ to get
$$
0 <\frac{ \|[\theta,\psi]\|_{H}}{2\lambda(k+1)} \leq e_s(A).
$$
This completes the proof of Theorem~\ref{theorem:displacement-positive}.
\end{proof}

\begin{remark}
The proof of Theorem~\ref{theorem:eliashberg-polterovich} relied on the bi-invariance of 
the distance $d_H$, whereas the proof of Theorem~\ref{theorem:displacement-positive} relied
on the equivalence of the norms $h \mapsto \|\phi \circ h \circ \phi^{-1}\|_{HL}$ and
$\|\cdot\|_{HL}$, i.e. the invariance of $d_{HL}$ up to a constant.
\end{remark}

%%%%%%%%%%%%%%%%%%%%%%%%%%%%%%%%%%%%%%%%%%%%%%%%%%%%%%%

\section{Examples}

A harmonic 1-parameter group is an isotopy $\Phi = \{\phi_t\}$ generated by the vector field 
$V_{\mathcal{H}}$ defined by $i(V_{\mathcal{H}})\omega = \mathcal{H}$, where $\mathcal{H}$
is a harmonic 1-form. It is immediate from the definitions that
$$
l_0(\Phi) = l_0(\Phi^{-1}) = |\mathcal{H}|
$$
where $|\cdot|$ is a norm on the space of harmonic 1-forms.
Hence  $l(\Phi) = |\mathcal{H}|$. Therefore, if $\phi_1$ is the time one map of $\Phi$ we have
$$
\|\phi_1\|_{HL} \leq |\mathcal{H}|.
$$

For instance, take the torus $T^{2n}$ with coordinates $(\theta_1,\ldots ,\theta_{2n})$ and the flat Riemannian metric. Then all the 1-forms $d\theta_i$ are harmonic. Given $v = (a_1, \ldots ,a_n, b_1, \ldots , b_n) \in \R^{2n}$, the translation $x \mapsto x+v$ on $\R^{2n}$ induces a rotation $\rho_v$ on $T^{2n}$, which is a symplectic diffeomorphism.
Moreover, $x \mapsto x + tv$ on $\R^{2n}$ induces a harmonic 1-parameter group $\{\rho_v^t\}$
on $T^{2n}$. 

Taking the 1-forms $d\theta_i$ for $i=1, \ldots ,2n$ as basis for the space of harmonic 1-forms and using the standard symplectic form
$$
\omega = \sum_{j=1}^n d\theta_j \wedge d\theta_{j+n}
$$ 
on $T^{2n}$ we have
$$
i(\dot{\rho_v^t})\omega = \sum_{j=1}^n \left(a_j d\theta_{j+n} - b_j d\theta_j\right).
$$
Thus,
$$
l(\{\rho_v^t\}) = |(-b_1,\ldots, -b_n,a_1,\ldots, a_n)|
$$
where $|\cdot|$ is a norm on the space of harmonic 1-forms,
and we see that
$$
\|\rho_v\|_{HL} \leq |v|
$$
if we use $|v| = |a_1|+\cdots + |a_n| + |b_1| + \cdots + |b_n|$
as the norm on both $\mathbb{R}^{2n}$ and the space of harmonic 1-forms.

\smallskip
Consider the torus $T^2$ as the square:
\[
 \{(p,q) \mid 0\leq p \leq 1 \text{ and } 0\leq q \leq 1\} \subset \R^2
\]
with opposite sides identified. For any $r < \frac{1}{2}$ let 
\[ 
\tilde{A}(r) = \{(x,y) \mid 0\leq x < r\} \subset \mathbb{R}^2,
\] 
and let $A(r)$ be the corresponding subset in $T^2$. If $v = (a_1,0)$ with $r \leq a_1 \leq 1-r$,
then the rotation $\rho_v$ induced by the translation $(p,q) \mapsto (p + a_1, q)$ displaces $A(r)$.
Therefore, using the norm $|v| = |a_1| + |b_1|$ we have
$$
\|\rho_a\|_{HL} \leq l(\{\rho_a^t\}) = a_1.
$$
Since this holds for all $r \leq a_1$ we have,
$$
e_s(A(r)) \leq r.
$$

\begin{remark}
Note that in the above example the symplectic displacement energy is finite,
whereas the Hamiltonian displacement energy $e(A(r))$ is infinite.  This follows
from a result proved by Gromov \cite{GroPse}: If $(M,\omega)$ is a symplectic manifold
without boundary that is convex at infinity and $L\subset M$ is a compact Lagrangian
submanifold such that $[\omega]$ vanishes on $\pi_2(M,L)$, then for any Hamiltonian
symplectomorphism $\phi:M \rightarrow M$ the intersection $\phi(L)\cap L \neq
\emptyset$. Stronger versions of this result can be found in \cite{FloMor}, \cite{FloThe},
and \cite{FloCup}. See also Section 9.2 of \cite{McDJ-h}.
\end{remark}

%%%%%%%%%%%%%%%%%%%%%%%%%%%%%%%%%%%%%%%%%%%%%%%%%%%%%%%%%%%%%%%%%%%%%%%%%

\section{Application}

The following result is an immediate consequence of the positivity of the symplectic displacement
energy of non-empty open sets. For two isotopies $\Phi$ and $\Psi$ denote by $\Phi^{-1} \circ \Psi$ the
isotopy given at time $t$ by $(\Phi^{-1} \circ \Psi)_t = \phi^{-1}_t \circ \psi_t$.

\begin{theorem}\label{theorem:uniqueness}
Let $\Phi_n$ be a sequence of symplectic isotopies and let $\Psi$ be another symplectic isotopy.
Suppose that the sequence of time-one maps $\phi_{n,1}$ of the isotopies $\Phi_n$ converges uniformly
to a homeomorphism $\phi$, and $l(\Phi_n^{-1} \circ \Psi) \to 0$ as $n \to \infty$, then 
$\phi = \psi_1$.
\end{theorem}

This theorem can be viewed as a justification for the following,
which appeared in \cite{BanTheg} and \cite{BanOnt}.

\begin{definition}
A homeomorphism $h$ of a compact symplectic manifold is called a {\bf strong symplectic homeomorphism} if there exist a sequence 
$\Phi_n$ of symplectic isotopies such that $\phi_{n,1}$ converges uniformly to $h$, and $l(\Phi_n)$ is a Cauchy sequence.
\end{definition}

\begin{proof}[Proof of Theorem~\ref{theorem:uniqueness}]
Suppose $\phi \neq \psi_1$, i.e. $\phi^{-1} \circ \psi_1 \neq \id$.
Then there exists a small open ball $B$ such that $(\phi^{-1} \circ \psi_1) (B) \cap B = \emptyset$.
Since $\phi_{n,1}$ converges uniformly to $\phi$, $((\phi_{n,1})^{-1} \circ \psi_1)(B) \cap B = \emptyset$ 
for $n$ large enough.
Therefore, the symplectic energy $e_s(B)$ of $B$ satisfies 
\[
e_s(B) \leq \|(\phi_{n,1})^{-1} \circ \psi_1 \|_{HL} \leq l(\Phi_n^{-1} \circ \Psi).
\]
The later tends to zero, which contradicts the positivity of $e_s(B)$.
\end{proof}

\begin{remark}
This theorem was first proved by Hofer and Zehnder for $M = \mathbb{R}^{2n}$ \cite{HofSym}, and then by 
Oh-M{\"u}ller in \cite{OhThe} for Hamiltonian isotopies using the same lines as above, and very 
recently by Tchuiaga \cite{TchOns}, using the $L^\infty$ version of the Hofer-like norm.
\end{remark}

\bigskip\noindent
{\bf Acknowledgments}

\noindent
We would like to thank the referee for carefully reading an earlier version of this paper
and providing the statement of Theorem \ref{conjinequality}.

%%%%%%%%%%%%%%%%%%%%%%%%%%%%%%%%%%%%%%%%%%%%%%%%%%%%%%%%%%%%%%%%%%%%%%%%

\nocite{BanOnt}
\nocite{BanSur}
\nocite{BanThe}
\nocite{EliBii}
\nocite{HofSym}
\nocite{WarFou}

\bibliography{books,papers}

\def\cprime{$'$}
\begin{bibdiv}
\begin{biblist}

\bib{BanTheg}{article}{
      author={Banyaga, Augustin},
      author={Tchuiaga, St{\'e}phane},
       title={The group of strong symplectic homeomorphisms in the {$L\sp
  \infty$}-metric},
        date={2014},
        ISSN={1615-715X},
     journal={Adv. Geom.},
      volume={14},
      number={3},
       pages={523\ndash 539},
         url={http://dx.doi.org/10.1515/advgeom-2013-0041},
      review={\MR{3228898}},
}

\bib{BanSur}{article}{
      author={Banyaga, Augustin},
       title={Sur la structure du groupe des diff\'eomorphismes qui
  pr\'eservent une forme symplectique},
        date={1978},
        ISSN={0010-2571},
     journal={Comment. Math. Helv.},
      volume={53},
      number={2},
       pages={174\ndash 227},
         url={http://dx.doi.org/10.1007/BF02566074},
      review={\MR{490874 (80c:58005)}},
}

\bib{BanThe}{book}{
      author={Banyaga, Augustin},
       title={The structure of classical diffeomorphism groups},
      series={Mathematics and its Applications},
   publisher={Kluwer Academic Publishers Group},
     address={Dordrecht},
        date={1997},
      volume={400},
        ISBN={0-7923-4475-8},
      review={\MR{98h:22024}},
}

\bib{BanOnt}{article}{
      author={Banyaga, Augustin},
       title={On the group of symplectic homeomorphisms},
        date={2008},
        ISSN={1631-073X},
     journal={C. R. Math. Acad. Sci. Paris},
      volume={346},
      number={15-16},
       pages={867\ndash 872},
         url={http://dx.doi.org/10.1016/j.crma.2008.06.011},
      review={\MR{2441923 (2009f:53137)}},
}

\bib{BanAHo}{incollection}{
      author={Banyaga, Augustin},
       title={A {H}ofer-like metric on the group of symplectic
  diffeomorphisms},
        date={2010},
   booktitle={Symplectic topology and measure preserving dynamical systems},
      series={Contemp. Math.},
      volume={512},
   publisher={Amer. Math. Soc.},
     address={Providence, RI},
       pages={1\ndash 23},
         url={http://dx.doi.org/10.1090/conm/512/10057},
      review={\MR{2605311 (2011d:53213)}},
}

\bib{BooTra}{article}{
      author={Boothby, William~M.},
       title={Transitivity of the automorphisms of certain geometric
  structures},
        date={1969},
        ISSN={0002-9947},
     journal={Trans. Amer. Math. Soc.},
      volume={137},
       pages={93\ndash 100},
      review={\MR{0236961 (38 \#5254)}},
}

\bib{BusPse}{article}{
      author={Buss, Guy},
      author={Leclercq, R{\'e}mi},
       title={Pseudo-distances on symplectomorphism groups and applications to
  flux theory},
        date={2012},
        ISSN={0025-5874},
     journal={Math. Z.},
      volume={272},
      number={3-4},
       pages={1001\ndash 1022},
         url={http://dx.doi.org/10.1007/s00209-011-0970-z},
      review={\MR{2995152}},
}

\bib{EliBii}{article}{
      author={Eliashberg, Yakov},
      author={Polterovich, Leonid},
       title={Bi-invariant metrics on the group of {H}amiltonian
  diffeomorphisms},
        date={1993},
        ISSN={0129-167X},
     journal={Internat. J. Math.},
      volume={4},
      number={5},
       pages={727\ndash 738},
         url={http://dx.doi.org/10.1142/S0129167X93000352},
      review={\MR{1245350 (94i:58029)}},
}

\bib{FloMor}{article}{
      author={Floer, Andreas},
       title={Morse theory for {L}agrangian intersections},
        date={1988},
        ISSN={0022-040X},
     journal={J. Differential Geom.},
      volume={28},
      number={3},
       pages={513\ndash 547},
      review={\MR{90f:58058}},
}

\bib{FloThe}{article}{
      author={Floer, Andreas},
       title={The unregularized gradient flow of the symplectic action},
        date={1988},
        ISSN={0010-3640},
     journal={Comm. Pure Appl. Math.},
      volume={41},
      number={6},
       pages={775\ndash 813},
      review={\MR{89g:58065}},
}

\bib{FloCup}{article}{
      author={Floer, Andreas},
       title={Cuplength estimates on {L}agrangian intersections},
        date={1989},
        ISSN={0010-3640},
     journal={Comm. Pure Appl. Math.},
      volume={42},
      number={4},
       pages={335\ndash 356},
         url={http://dx.doi.org/10.1002/cpa.3160420402},
      review={\MR{990135 (90g:58034)}},
}

\bib{GroPse}{article}{
      author={Gromov, M.},
       title={Pseudoholomorphic curves in symplectic manifolds},
        date={1985},
        ISSN={0020-9910},
     journal={Invent. Math.},
      volume={82},
      number={2},
       pages={307\ndash 347},
      review={\MR{87j:53053}},
}

\bib{HofSym}{book}{
      author={Hofer, Helmut},
      author={Zehnder, Eduard},
       title={Symplectic invariants and {H}amiltonian dynamics},
      series={Modern Birkh\"auser Classics},
   publisher={Birkh\"auser Verlag},
     address={Basel},
        date={2011},
        ISBN={978-3-0348-0103-4},
         url={http://dx.doi.org/10.1007/978-3-0348-0104-1},
        note={Reprint of the 1994 edition},
      review={\MR{2797558 (2012b:53191)}},
}

\bib{HofOnt}{article}{
      author={Hofer, Helmut},
       title={On the topological properties of symplectic maps},
        date={1990},
        ISSN={0308-2105},
     journal={Proc. Roy. Soc. Edinburgh Sect. A},
      volume={115},
      number={1-2},
       pages={25\ndash 38},
         url={http://dx.doi.org/10.1017/S0308210500024549},
      review={\MR{1059642 (91h:58042)}},
}

\bib{McDJ-h}{book}{
      author={McDuff, Dusa},
      author={Salamon, Dietmar},
       title={{$J$}-holomorphic curves and quantum cohomology},
      series={University Lecture Series},
   publisher={American Mathematical Society},
     address={Providence, RI},
        date={1994},
      volume={6},
        ISBN={0-8218-0332-8},
      review={\MR{95g:58026}},
}

\bib{OhThe}{article}{
      author={Oh, Yong-Geun},
      author={M{\"u}ller, Stefan},
       title={The group of {H}amiltonian homeomorphisms and {$C\sp
  0$}-symplectic topology},
        date={2007},
        ISSN={1527-5256},
     journal={J. Symplectic Geom.},
      volume={5},
      number={2},
       pages={167\ndash 219},
         url={http://projecteuclid.org/getRecord?id=euclid.jsg/1202004455},
      review={\MR{2377251 (2009k:53227)}},
}

\bib{PolThe}{book}{
      author={Polterovich, Leonid},
       title={The geometry of the group of symplectic diffeomorphisms},
      series={Lectures in Mathematics ETH Z\"urich},
   publisher={Birkh\"auser Verlag},
     address={Basel},
        date={2001},
        ISBN={3-7643-6432-7},
         url={http://dx.doi.org/10.1007/978-3-0348-8299-6},
      review={\MR{1826128 (2002g:53157)}},
}

\bib{TchOns}{article}{
      author={Tchuiaga, St{\'e}phane},
       title={On symplectic dynamics},
        date={2013},
     journal={preprint},
}

\bib{WarFou}{book}{
      author={Warner, Frank~W.},
       title={Foundations of differentiable manifolds and {L}ie groups},
      series={Graduate Texts in Mathematics},
   publisher={Springer-Verlag},
     address={New York},
        date={1983},
      volume={94},
        ISBN={0-387-90894-3},
      review={\MR{722297 (84k:58001)}},
}

\end{biblist}
\end{bibdiv}
\bibliographystyle{amsalpha}
\end{document}